\documentclass[12pt,leqno,a4paper]{amsart}
\usepackage{amssymb, amsmath, enumerate}
\usepackage{enumerate}
\usepackage[usenames]{color}

\newtheorem{theorem}{Theorem}[section]
\newtheorem{lemma}[theorem]{Lemma}

\newtheorem*{thmA}{Theorem A}
\newtheorem*{thmB}{Theorem B}
\newtheorem*{thmC}{Theorem C}

\theoremstyle{remark}
\newtheorem{remark}[theorem]{Remark}

\numberwithin{equation}{section}

\begin{document}
\title[Extensions of a theorem of P. Hall on indexes of maximal subgroups]{Extensions of a theorem of P. Hall on indexes of maximal subgroups}
    \author[Beltr\'an and Shao]{
      Antonio Beltr\'an\\
     Departamento de Matem\'aticas\\
      Universitat Jaume I \\
     12071 Castell\'on\\
      Spain\\
      \\Changguo Shao \\
College of Science\\ Nanjing University of Posts and Telecommunications\\
     Nanjing 210023 Yadong\\
      China\\
     }

 \thanks{Antonio Beltr\'an: abeltran@uji.es ORCID ID: https://orcid.org/0000-0001-6570-201X \newline \indent
Changguo Shao: shaoguozi@163.com ORCID ID: https://orcid.org/0000-0002-3865-0573}



\begin{abstract}
We extend a classical theorem of P. Hall that claims that if the index of every maximal subgroup of a finite group $G$ is a prime or the square of a prime,
then $G$  is solvable.
Precisely, we prove that if one allows, in addition, the possibility that every  maximal subgroup of $G$ is nilpotent instead of having prime or squared-prime index, then $G$ continues to be solvable.
  Likewise, we obtain the solvability of $G$ when we assume that every proper non-maximal subgroup of $G$ lies in some subgroup of index prime or squared prime.
\end{abstract}

\maketitle

\section{Introduction}

A well-known theorem of P. Hall  \cite[9.4.1]{Robinson} establishes that if the index of every maximal subgroup of a finite group $G$ is a prime number or the square of a prime number,
 then $G$ is solvable. First, we extend this result by allowing maximal subgroups of $G$ the possibility of being nilpotent too, so we obtain the following  solvability criterion.

\begin{thmA} If every maximal subgroup of a finite group $G$ either is nilpotent or has index a
prime or the square of a prime, then $G$ is solvable.
\end{thmA}

We note that Theorem A also extends  other recent results \cite{LPZ,YJK} that demonstrate, among other properties,
 the solvability of those groups whose non-nilpotent maximal subgroups have prime index. Our proof of Theorem A is quite elementary;
 it basically utilizes the $p$-nilpotency criterion of Glauberman and Thompson  for an odd prime $p$ \cite[Satz IV.6.2]{Hup}.

\medskip
It is worth mentioning that another extension of Hall's theorem was given in \cite{MT}.  It was proved that when every proper non-maximal
 subgroup of a  group $G$  lies in a subgroup of prime index, then $G$ is solvable. Indeed, it turns out that $G/{\bf F}(G)$ is supersolvable, where ${\bf F}(G)$ denotes the Fitting subgroup of $G$. We go a step further, 
 by obtaining the following solvability criterion that also generalizes Hall's Theorem.

\begin{thmB} Let $G$ be a finite group. If every proper non-maximal subgroup of  $G$ is contained in some subgroup whose index is a prime or the square of a prime, then $G$ is solvable.
\end{thmB}

The proof of Theorem B, however,  relies on several results depending on the Classification of Finite Simple Groups,
among others,  Guralnick's classification of simple groups having a prime power index subgroup.  We also need to do a detailed analysis on certain simple groups
that requires the knowledge of the  structure of maximal subgroups
  of several nonabelian simple groups, for which we will appeal to different sources \cite{BHR, Con, Hup, Liebeck2}.

  \medskip
 We also remark that there are no more ways to extend Theorem B in the sense that no more possible prime-powers can be included in the assumptions of Theorem B so as to get a new solvability criterion.
It suffices to  take into account that every maximal subgroup of
 the simple group ${\rm PSL}_2(7)$ has index $7$ or $8$, so every proper non-maximal subgroup lies in a subgroup whose index is prime or the cube of a prime.

  \medskip
 On the other hand, for proving Theorem B we first  need to establish a classification of nonabelian simple groups having a squared prime index subgroup.
 This result has interest on its own and we are not aware that has already been published.

 \begin{thmC} Let $S$ be a finite nonabelian simple group that has a subgroup of index $q^2$ with $q$ prime. Then one of the following holds
\begin{enumerate}
\item[(a)] $S \cong {\rm Alt}(q^2)$ with subgroups of index $q^2$ $(q\geq 3)$, or
\item[(b)] $S\cong {\rm PSL}_2(8)$ with subgroups of index $3^2$, or
\item[(c)] $S\cong {\rm PSL}_5(3)$ with  subgroups of index $11^2$.
\end{enumerate}
\end{thmC}

\medskip
Throughout this paper, all groups are supposed to be finite, and we follow standard notation (e.g. \cite{Robinson}).

\section{Preliminaries}

We start with an application of Glauberman-Thompson's $p$-nilpotency criterion for an odd prime $p$ mentioned in the Introduction, which is useful for our purposes.
We stress, however, that the following lemma is not true for $p=2$: the Mathieu group $M_{10}\cong {\rm Alt}(6).2$, with a  minimal normal subgroup isomorphic to ${\rm Alt}(6)$,
 which is maximal too,  is an example of it because every maximal subgroup of $M_{10}$ distinct from ${\rm Alt}(6)$ is $2$-nilpotent.

\begin{lemma}\label{maximal} Let $G$ be a finite group and $N$  minimal normal subgroup of $G$ that is not solvable. Let $p$ be any odd prime dividing $|N|$ and $P$ a Sylow $p$-subgroup of $N$. Then there exists a maximal subgroup in $G$ such that contains ${\bf N}_G(P)$, does not contain $N$, and is not $p$-nilpotent.
\end{lemma}

\begin{proof}
Notice that $N$ must be a direct product of isomorphic nonabelian simple groups. Let $p$ an odd prime divisor of $|N|$ and let $P\neq 1$ be a Sylow $p$-subgroup of $N$. Let $J(P)\neq 1$ be the Thompson subgroup of $P$ (see  \cite[IV.6.1.]{Hup} for a definition). Clearly ${\bf Z}(J(P))$ is not normal in $G$ since cannot be normal in $N$, and accordingly, ${\bf N}_G({\bf Z}(J(P))) $ is a  proper
 subgroup of $G$ lying in some maximal subgroup, say $M$, of $G$.  Now, if $M$  is $p$-nilpotent,  so is ${\bf N}_N({\bf Z}(J(P)))$. By Glauberman-Thompson's criterion,  $N$ is $p$-nilpotent too (with $p$ dividing $|N|$), which contradicts the non-solvability of $N$. Hence, $M$ is not $p$-nilpotent. On the other hand, the Frattini argument gives ${\bf N}_G(P)N=G$, and moreover, since ${\bf Z}(J(P))$ is characteristic in $P$, we have
 $$ {\bf N}_G(P)\leq {\bf N}_G({\bf Z}(J(P)))\leq M,$$
 so we conclude that $M$ does not contain $N$ either. Thus, the lemma is proved.
 \end{proof}

We need the following property concerning the structure of certain  subgroups of direct products of nonabelian simple groups.

\begin{lemma}\label{m2}\cite[Lemma 2.4]{Q} Let $G$ be a finite group and $N$ a minimal normal subgroup of $G$ such that $N=S_1\times \cdots\times S_t,$ with $ t\geq 1$ and $S_i$ are isomorphic nonabelian simple groups. Assume that $H$ is a maximal
subgroup of $G$ such that $G= HN$. Set $K=H\cap N$. Then one of the following holds:
\begin{itemize}
\item[(1)] $K=K_1\times \cdots\times K_t$, where $K_i< S_i$ for all $i$, and $H$ acts transitively on the
set $\{K_1,\cdots, K_t\}$.

\item[(2)] $K=A_1\times \cdots\times A_k$ is minimal normal in $H$, where $A_1\cong \cdots\cong A_k\cong S_1$, $k\mid t$ and $k<t$,
  and every nonidentity element of $A_i$ has length $t/k$ under the decomposition.
  \end{itemize}
\end{lemma}

Our proofs are based on the classic Guralnick's classification mentioned in the Introduction, which we state here for the reader's convenience.

\begin{lemma} {\normalfont  \cite[Theorem 1]{Gura}} \label{g}
\ Let $G$ be a nonabelian simple group with $H < G$ and
$|G : H| = p^a$, $p$ prime. One of the following holds.

\begin{itemize}

\item[$(a)$] $G={\rm Alt}(n)$, and $H\cong {\rm Alt}(n-1)$, with $n= p^a$.

\item[$(b)$] $G = {\rm PSL}_n(q)$ and $H$ is the stabilizer of a line or hyperplane. Then
$|G: H|= (q^n-1)/(q - 1) = p^a$. (Note that $n$ must be prime).

\item[$(c)$] $G={\rm PSL}_2(11)$ and $H\cong {\rm Alt}(5)$.

\item[$(d)$] $G=M_{23}$ and $H\cong M_{22}$ or $G=M_{11}$ and $H\cong M_{10}$.

\item[$(e)$] $G= {\rm PSU}_4(2)\cong {\rm PSp}_4(3)$ and $H$ is the parabolic subgroup of
index $27$.
\end{itemize}
\end{lemma}

Finally, we state another auxiliary result, which is quite well known and also relies on the Classification.

\begin{lemma}{\normalfont \cite[Proposition]{BGH}} \label{out} Let $S$ be  a nonabelian simple group. Then the greatest prime divisor of $|S|$ does not divide $|{\rm Out}(S)|$.
\end{lemma}

\section{Proofs}

\begin{proof}[Proof of Theorem A]

We argue by  minimal counterexample, and suppose that $G$ is a non-solvable group of minimal order satisfying the assumptions.
 The hypotheses are clearly inherited by  quotients,  so it can be assumed that $G$ possesses exactly one minimal normal subgroup, say $N$, which is not solvable (possibly $N=G$).
   As we said,  $G/N$ is solvable by minimality, and then $N$ is a direct product of isomorphic nonabelian simple groups.

\medskip
 We take $p$ to be the largest (odd) prime dividing $|N|$ and $P$ a Sylow $p$-subgroup of $N$.
  We apply Lemma \ref{maximal} to $N$ and get that there exists a non-$p$-nilpotent maximal subgroup of $G$, say $M$, such that $N\not\leq M$ and $L:={\bf N}_G(P)\leq M$.  Hence, by hypothesis  $M$ has prime  or squared-prime index in $G$, say $|G:M|=q$ or $q^2$.  In addition, we have $G=NM$ and then $|G:M|=|N: N\cap M|$ , from which we deduce that $q$ divides $|N|$, and consequently $q \leq p$ (a priori possibly equal).
Next, the conjugates of $P$ in $G$ account for all the Sylow $p$-subgroups of
$N$; therefore $|G:L| \equiv 1$ (mod $p)$ by Sylow's Theorem. By the same reason
$|M: L | = 1$ (mod $p)$, and it follows that $|G : M|\equiv 1$ (mod $p)$. This forces $q \not\equiv 1$ (mod $p)$
because $q \leq p$, so we are left
with the possibility $|G:M| = q^2 \equiv 1$ (mod $p )$ and thus, $q \equiv -1$ (mod $p)$. However, this is only possible when $p = 3$
and $q = 2$. Consequently, $| N : N \cap M| = 4$ and hence $N$ has as an image some
nontrivial subgroup of the symmetric group ${\rm Sym}(4)$, contradicting the non-solvability of $N$. This implies that the counterexample does not exist.
\end{proof}

Before proving Theorem B, we need to prove Theorem C, which,  as said in the Introduction, has own interest.

\begin{proof}[Proof of Theorem C]
In view of  Lemma \ref{g} we only have two possibilities for $S$: $S \cong {\rm Alt}(q^2)$ or $S\cong {\rm PSL}_n(s)$ with $(s^n-1)/(s-1)=q^2\geq 9$ (with $n$ prime).
 Thus, we only have to analyze the second case.
Recall that the order of the Singer cycle of ${\rm PSL}_n(s)$ is equal to $$\frac{s^n-1}{(n,s-1)(s-1)}=\frac{q^2}{(n,s-1)},$$ (see  \cite[2.7.3]{Hup}) and this is an integer.
As $n$ is prime, this yields two possibilities: either  $(n,s-1)=q$ or $1$.   Assume first $(n,s-1)=q$. Then $q$ divides $(s-1)$, so in particular  $q^2< s^2$.
If $n> 2$, then the fact that $q^2= s^{n-1}+ \ldots +s+1$ provides a contradiction. If $n=2$, then $q^2= s+1$, and hence $q$ divides $(s+1,s-1)=1$ or $2$, which is also a contradiction.
Assume now the second possibility, $(n,s-1)=1$, and write $s=r^d$, where $r$ is  a prime and $d$ is a positive integer.
Now, if $n=2$, then  $q^2= r^d+1$, and since $q$ is odd this  trivially implies that $d>1$.
Then, it is well known  (Catalan's conjecture) that the only solution of such  equation is $q=3$, $r=2$ and $d=3$, that is, $S\cong {\rm PSL}_2(8)$, so we obtain (b).
On the other hand, if $n> 2$, then the diophantine equation $$\frac{s^{n}-1}{s-1}= q^2$$ is a particular case of the so-called Nagell-Ljunggren equation (see \cite{BM}).
Then, as $q$ is a prime number, it is known that there is exactly  one solution (see \cite[Theorem NL]{BM}); this is $n=5, s=3$ and $q=11$. This means that $S\cong {\rm PSL}_5(3)$, so the proof is finished.
\end{proof}

\begin{proof}[Proof of Theorem B]

 We argue by  minimal counterexample, and suppose that $G$ is a non-solvable group of minimal order satisfying the hypotheses.
 Take $N$ to be a minimal normal subgroup of $G$. The hypotheses are certainly inherited by quotients,
   so $G/N$ is solvable by minimality.
   Furthermore, it can be assumed that $G$ possesses exactly one minimal normal subgroup, say $N$, which is non-solvable (possibly $N=G$).
    Then $N$ is a direct product of isomorphic nonabelian simple groups.

 \medskip
 Next we claim that there exists a maximal subgroup $M$ of $G$ such that $|G:M|$ is prime or a squared prime and $N\not\subseteq M$.
 Let $p\geq 5$ be  a prime divisor of $|N|$, and take $P$ to be a Sylow $p$-subgroup of $N$. Then the Frattini argument gives $G={\bf N}_G(P)N$.
 Assume first that ${\bf N}_G(P)$ is not maximal in $G$. Then the hypotheses imply  that ${\bf N}_G(P)$ lies in is some subgroup $M$ of $G$ of index $q$ or $q^2$ for some prime $q$.
   Without loss we can assume that $M$ is maximal in $G$, otherwise, there exists  a maximal subgroup $M_1$ of $G$, with $M\leq M_1$  and $|G:M_1|=|M_1:M|=q$,
    and then we can replace $M$ by $M_1$. On the other hand, it is clear that $N\not\subseteq M$, so the claim is proved.
Suppose now that ${\bf N}_G(P)$ is maximal in $G$.
 Assume  first that ${\bf N}_N(P)\subseteq \Phi({\bf N}_G(P))$. In particular, we have that ${\bf N}_N(P)$ is nilpotent.
  Since $p>3$, then Thompson's theorem \cite[Theorem X.8.13]{HB} implies that ${\bf O}^p(N)<N$, which is a contradiction.
   Therefore the assumption is false, and there exists a maximal subgroup $L$ of ${\bf N}_G(P)$ such that ${\bf N}_N(P)\not \subseteq L$.
    It follows that ${\bf N}_G(P)=L{\bf N}_N(P)$, and hence $G={\bf N}_G(P)N=LN$.
    However, $L$ is not maximal in $G$, so by hypothesis there exists some  subgroup $M$ of $G$ such that $L\subseteq M$ and $M$ has prime or squared prime index.
     As above, there is no loss if we can assume that $M$ is maximal in $G$, and since $G=MN$, then $N\not\subseteq M$. Therefore, the claim is proved.

\medskip
 Henceforth, we take $M$ to be a maximal subgroup of $G$ with $N\nleq M$ and  the index of $M$ is prime, say  $q$, or the square of a prime, say $q^2$.
   The rest of the proof consists in proving  that this leads to a contradiction.

 \medskip
Notice that $G=MN$. Write $N=S_1\times \cdots\times S_r$, where $S_i$ are isomorphic nonabelian simple groups and note that $M$ acts transitively on $\{S_1,\ldots,S_r\}$.  It is clear that  $S_i\nleq M$ for every $i$. Let $N_0:=M\cap N$. By  applying Lemma \ref{m2}, we obtain $N_0=K_1\times \cdots\times K_r$, where $K_i < S_i$ and $M$ acts transitively on the set  $\{K_1,\cdots, K_r\}$, because   Lemma \ref{m2}(2)  leads to that $|S_i|$ divides $|G:M|$, and this is not possible. This implies that $|G:M|=|N:N_0|=|S_i:K_i|^r$ and this forces $r\leq 2$. Accordingly, we distinguish two cases: $r=2$ and $r=1$.

\medskip
  (a) Assume first that $r=2$ and write $N_0=K_1\times K_2$. Then $|S_i:K_i|=q$ and $S_i/{\rm core}_{S_i}(K_i)\leq {\rm Sym}(q)$, and hence, it follows that  $q$ does not divide $|K_i|$ and $q$ is the largest prime divisor of $|S_i|$. This shows that if the maximal subgroup of $G$ has index the square of a prime, then the prime is the largest prime divisor of $|N|$.  Thus, $q\geq 5$.
 Moreover, by the uniqueness of $N$, we have  ${\bf C}_G(N) =1$, and  accordingly $N<G\leq {\rm Aut}(N)$.
 Hence $G\leq ({\rm Aut}(S_1)\times {\rm Aut}(S_2))\rtimes A$, where $A$ permutes ${\rm Aut}(S_i)$, and is cyclic of order 2. We remark that $A\leq G$ since $G$ acts transitively on $\{S_1, S_2\}$.
  Since $q$ is the largest prime divisor of  $|S_i|$, by applying Lemma \ref{out},
   we deduce that $G/N$ is a $q'$-group. As a consequence, if  $Q$ a Sylow $q$-subgroup of $S_1$, then $Q\times Q$ is a Sylow $q$-subgroup of $G$.
   Let us consider $(Q\times Q )\rtimes A $, which  is    not maximal in $G$ because $Q <{\bf N}_{S_1} (Q)$ (this follows again by \cite[Theorem X.8.13]{HB})
  as in the second paragraph of the proof).
      Then, by hypothesis,  $(Q\times Q )\rtimes A $   should be contained in a maximal subgroup of $G$, say $L$, with prime index or squared prime index
    However, we also have $LN=G$ and then,  we can argue as above to get that only case (1) of Lemma \ref{m2} can happen, that is,
    $L\cap N= H_1\times H_2$ with $H_i<S_i$ and $|G:L| = |N:L\cap N | = |S_1:H_1|^2=r^2$ for some prime $r$. Of course $r<q$ since $Q\times Q\leq L$. Now, the simplicity of $S_i$
    combined with the fact that $S_i$ has a proper subgroup of index less than $q$ clearly leads to a contradiction.

\medskip
(b) Suppose that $r=1$, that is, $N$ is a nonabelian simple group and $N\leq G\leq {\rm Aut}(N)$. In addition, we know that $|G:M|=|N:N\cap M|$ is a prime or a square of a prime.

\medskip
(b.1) Assume first that $|N:N\cap M|=q$ is prime. Again we have that $q$ is the largest prime divisor of $|N|$ and $G/N$ is a $q'$-group by Lemma \ref{out}.
 According to Lemma \ref{g}, we only have to analyze each of these cases: $N\cong {\rm Alt}(q);  {\rm PSL}_n(q_0)$ with $(q_0^n -1)/(q_0 -1)$ prime;   $ {\rm PSL}_2(11)$;   $M_{23}$  and $M_{11}$.

\medskip

 For the case $N\cong {\rm Alt}(q)$ with $q\geq 5$, we have $G\cong {\rm Alt}(q)$ or  $G\cong {\rm Sym}(q)$.
  In the first case, it is clear that any Sylow $q$-subgroup of $G$ is not maximal in $G$ and is not contained in any subgroup of prime or squared prime index.
  For $G={\rm Sym}(q)$, however, a Sylow $q$-subgroup $Q$ of $G$ does not work since it is contained in $N={\rm Alt}(q)$ of index $2$. We proceed as follows instead.
  By the Frattini argument we have $G={\bf N}_{G}(Q)N$ and $|{\bf N}_G(Q):{\bf N}_N(Q)|=2$.     In fact, if we write $G=NT$ with $|T|=2$,
  then we can write         ${\bf N}_G(Q) = {\bf N}_N(Q)  \rtimes T$.
Notice that $QT$ is not maximal in $G$ because $Q<{\bf N}_G(Q)\cong Q\rtimes C_{q-1}$. Also, if $QT$ were contained in some subgroup
$L\leq G$ of prime index or squared prime index, we would have $G=NL$, and consequently $|N: L\cap N|$ would be a prime or a squared prime. However,
the only subgroups of $N$  with such index are those  isomorphic to ${\rm Alt}(q-1)$, which obviously cannot  contain $Q$.
This contradiction shows that ${\rm Sym}(q)$ does not satisfy the conditions of the theorem, and hence it is discarded too.

\medskip
Suppose  $N\cong  {\rm PSL}_n(q_0)$ with $\frac{q_0^n-1}{q_0-1}=q$  (and $n$  a prime).
 By \cite[2.7.3]{Hup}, we know that $N$ has a Singer cycle of order $$\frac{q_0^n-1}{(q_0-1)(n,q_0-1)}=\frac{q}{(n,q_0-1)}.$$
 Then we get  $(n,q_0-1)=1$.  By \cite[Tables]{BHR} and \cite[Proposition 4.1.17]{Liebeck2}, we have $N\cap M=E_{q_0}^{n-1}:{\rm SL}_{n-1}(q_0)$ (we follow the notation of \cite{BHR}).
 Again by using \cite[Tables]{BHR} and \cite[Proposition 4.1.17]{Liebeck2}, we consider a generator $\phi$ of the field automorphisms group  of $N$  such that  $N\cap M$ is  $\phi$-invariant,
 and  the graph automorphism of $N$, say $\gamma$, so  $G\leq N\langle \phi, \gamma\rangle$.
 Now, if $N\langle \gamma\rangle \leq  G$, then,  by using \cite[Tables]{BHR} and \cite[Proposition 4.1.17]{Liebeck2} again,  we get  $(M\cap N)\langle \gamma\rangle\nleq M$.
 This implies that $2$ divides $|G:M|$ and thus, $q=2$, a contradiction.
 This shows that $G\leq N\langle\phi \rangle$. Therefore, there exists some $i\geq 1$ such that $G=N\langle \phi^i\rangle$.
 Write $\phi_1=\phi^i$, that is, $G=N\langle \phi_1 \rangle$.
 Let $H$ be a Singer cycle of $N$ such that ${\bf N}_N(H)$ is $\phi_1$-invariant (see \cite[Tables]{BHR} and \cite[Proposition 4.1.17]{Liebeck2} again).
 Then ${\bf N}_N(H)\langle \phi_1\rangle< G$ and hence $H \langle \phi_1 \rangle$ is not maximal in $G$.
By hypothesis, there exists a maximal subgroup $M_1$ of $G$ such that $H \langle \phi_1 \rangle< M_1$ and $|G:M_1|$ is a prime or a square of a prime.
 Since $M_1=(M_1\cap N)\langle \phi_1\rangle$, we have that   $|G:M_1|=|N:N\cap M_1|$  is a prime or square of a prime.
 By applying Lemma \ref{g}, we deduce that $|N:N\cap M_1|=\frac{q_0^n-1}{q_0-1}=q$, which is a contradiction.

\medskip
 If $N\cong {\rm PSL}_2(11)$, then Out$(N)\cong C_2$, and thus, $G\cong {\rm PSL}_2(11)$ or $G\cong {\rm PGL}_2(11)$. In the first case, the maximal subgroups of $G$,  apart from ${\rm Alt}(5)$ of index $11$,
  are isomorphic to $C_{11}\rtimes C_5$ or  to $D_{12}$. But no subgroup of order $11$ is  maximal  in $G$ and neither satisfies
  the conditions of the theorem, so this case is impossible.    Suppose now that $G={\rm PGL}_2(11)$ and let $\delta$ be the diagonal automorphisms of ${\rm PSL}_2(11)$, so  $G=N\langle \delta \rangle$. Let $P_{11}$ be a Sylow $11$-subgroup of ${\rm PSL}_2(11)$ that is $\delta$-invariant.  By \cite{Con}, we know that $P_{11}\langle \delta \rangle$ is not maximal in $G$.
  Therefore, there should exist a  maximal subgroup $U$ of $G$ such that $P_{11}\langle \delta \rangle< U$ and $|G:U|$  is a prime  or a square of a prime. But this also leads to a contradiction according to \cite{Con}.
\medskip

  Finally,  both sporadic groups have trivial outer automorphism group, so in both cases $G=N$.
  By using \cite{Con}, one can check  that in $M_{11}$ any subgroup of order $11$ is not contained
   in a subgroup of prime index or squared prime index (in fact, $M_{10}$ is the unique such subgroup, and has index $11$).
    Similarly, for $M_{23}$, every subgroup isomorphic to $C_{23}$ is  not maximal and does not satisfy our conditions.

\medskip

  (b.2) Assume now that  $|N:N\cap M|=q^2$ with $q$ prime.
  In view of  Theorem C, there are only  three cases to study:
   $N \cong {\rm Alt}(q^2), {\rm PSL}_2(8)$ or ${\rm PSL}_5(3)$.
   In the first case, it follows that  $G\cong  {\rm Alt}(q^2)$ or $ {\rm Sym}(q^2)$. Assume first that $G\cong{\rm Alt}(q^2)$ (with $q\geq 3$).
   It is immediate that every Sylow $q$-subgroup of $G$, which is not maximal, is not contained in any subgroup or prime index or squared prime index
   (because the unique such subgroup is ${\rm  Alt}(q^2-1)$).

\medskip
   For $G\cong {\rm Sym}(q^2)$, again we take
    $Q$ a Sylow $q$-subgroup of $G$,  which is also a Sylow $q$-subgroup of $N= {\rm Alt}(q^2)$.
By the Frattini argument we have $G={\bf N}_{G}(Q)N$ and
$|{\bf N}_G(Q):{\bf N}_N(Q)|=2$.     Indeed, if we write $G=NT$ where $|T|=2$, then we have  ${\bf N}_G(Q) = {\bf N}_N(Q)  \rtimes T$.
But notice that $QT$ is not maximal in $G$ because we know that  $Q<{\bf N}_G(Q)$. Moreover, if $QT$ were contained in some subgroup
$L$ of prime index or squared prime index, we would have $G=NL$, and consequently $|N: L\cap N|$ would be prime or squared prime. However,
the only subgroups of $N$  with such an index are those  isomorphic to ${\rm Alt}(q^2-1)$, which obviously cannot  contain $Q$.
This contradiction shows that ${\rm Sym}(q^2)$  can be discarded too.

\medskip
 Assume now $N\cong {\rm PSL}_2(8)$, whose outer automorphism groups has order $3$.
 Accordingly, we have  $G\cong {\rm PSL}_2(8)$ or $G \cong {\rm PSL}_2(8).3 \cong {\rm P\Gamma L}_2(8)$. By looking at \cite{Con},
  we easily check that both groups have Sylow 3-subgroups, which are not maximal subgroups, and  that are not contained in any subgroup of prime index  or squared prime index.
  Thus, this case cannot happen.

\medskip
Finally, suppose $N\cong {\rm PSL}_5(3)$.
 Its outer automorphism group has order 2, so $G= {\rm PSL}_5(3)$ or ${\rm PSL}_5(3).2$.
 By \cite[Table 8.18]{BHR} for instance, it is easily seen that the Sylow $11$-subgroups for both cases of  $G$ do not verify the index conditions of the theorem either.
 This final contradiction proves the solvability of $G$.
\end{proof}

\begin{remark}  As we said in the Introduction, when every proper non-maximal subgroup of a group $G$ lies in some subgroup of prime index, then $G/{\bf F}(G)$ is  supersolvable.
This does not occur under the assumptions of Theorem B. The affine special linear group $G={\rm ASL}(2,3)$ satisfies such conditions, indeed, the indexes of its maximal subgroups are 3, 4 and 9. 
However, $G/{\bf F}(G)\cong {\rm SL}_2(3)$, which is not supersolvable.
\end{remark}

\medskip
\noindent
{\bf Acknowledgements}

\medskip
This work is supported by the National Nature Science Fund of China (No. 12071181 and No. 12471017). A. Beltr\'an is also supported  by Generalitat Valenciana, Proyecto CIAICO/2021/193.
C.G. Shao is also supported by Natural Science Research Start-up Foundation of Recruiting Talents of Nanjing University of Posts and Telecommunications (Grant Nos. NY222090, NY222091).

 \bigskip
 \noindent
{\bf Data availability}
 Not applicable. Since this is a manuscript in theoretical mathematics, all the data we have used is within the papers listed in the
section References.

\bigskip
\noindent
{\bf Declarations}

\bigskip
\noindent
{\bf Conflict of Interest} The authors declare no competing interests.

\bibliographystyle{plain}

\end{document}